\documentclass[11pt]{amsart}
\usepackage[usenames]{color}
\usepackage{amsthm}
\usepackage{amssymb}
\newtheorem{thm}{{Theorem}}[section]
\newtheorem{lem}[thm]{Lemma}
\newtheorem{prop}[thm]{Proposition}
\newtheorem{defn}[thm]{Definition}
\newtheorem{cor}[thm]{Corollary}
\newtheorem{example}[thm]{Example}
\newcommand{\Z}{\mathbb{Z}}

\newcommand{\vv}{{\vec v}}
\newcommand{\vw}{{\vec w}}
\newcommand{\vn}{{\vec n}}
\newcommand{\vm}{{\vec m}}

\newcommand{\cA}{{\mathcal A}}

\newcommand{\cF}{{\mathcal F}}

\newcommand{\cH}{{\mathcal H}}

\newcommand{\cP}{{\mathcal P}}
\newcommand{\cQ}{{\mathcal Q}}
\newcommand{\cR}{{\mathcal R}}

\newcommand{\cT}{{\mathcal T}}
\newcommand{\BZ}{{\mathbb Z}}
\newcommand{\BR}{{\mathbb R}}
\newcommand{\BN}{{\mathbb N}}

\title[Directional
entropy, rank 1]{Rank One  $\mathbb Z^d$ Actions and Directional Entropy}
\date{September 10, 2008}
\author{E. Arthur Robinson, Jr.}
\address{Department of Mathematics\\ George Washington University\\ 2115 G St. NW\\ Washington, DC 20052}
\email{robinson@gwu.edu}

\author{Ay\c se A. \c Sah\. in}
\address{Department of Mathematical Sciences\\ DePaul University\\ 2320
North Kenmore Ave.\\ Chicago, IL 60614}
\email{asahin@condor.depaul.edu}

\keywords{rank $1$ actions, directional entropy}
\subjclass[2000]{37A15,37A35}

\thanks{The research of the second author was supported in part by a DePaul University Research Council Paid Leave.  The second author also wishes to thank the Department of Mathematics of the George Washington University for their hospitality during her sabbatical when most of this research was conducted. }
\begin{document}
\maketitle

\section{Introduction}

Rank one transformations play a central role in the theory of ergodic measure preserving transformations. Having first been identified as
a distinct class by Chacon in \cite{C}, their properties have been studied extensively (see for example \cite{Bax}, \cite{Fr}, \cite{King}, \cite{F1}). Rank one transformations have also served as an important tool for exploring the range of possible behavior of measure preserving transformations (see for example \cite{C}, \cite{Orn1}, \cite{Ru4}). The idea of rank one can be easily generalized to measure preserving actions of $\mathbb Z^d$.   Informally, we think of a rank one action of $\mathbb Z^d$ on a Lebesgue probability space $(X,\mathcal A,\mu)$ as a limit of actions defined on a sequence of Rohlin towers whose levels generate $\mathcal A$.  In the classical case, i.e the case $d=1$,
the most natural shapes for these towers are intervals.  Even in that
case, however, it is possible to define rank one more generally by allowing the towers to have more exotic shapes.   Ferenczi, for example, gives a definition of a class of transformations called {\it funny} rank one, in which the tower shapes are arbitrary F\o lner sets \cite{F}.   He shows that
the tower shapes do matter, because while all rank one transformations with interval tower shapes are loosely Bernoulli \cite{ORW}, there exists a funny rank one transformation that is not \cite{F}.


Many of the most basic properties of rank one transformations
extend to the $\mathbb Z^d$ case, $d>1$, with essentially no restrictions on the tower shapes. These include ergodicity, entropy zero, and simple spectrum   (see Section~\ref{s:rank1}).
While for $d>1$ it might appear that rectangular tower shapes are the obvious analogues of the interval tower shapes of the case $d=1$,
even within the class of  rank one $\mathbb Z^d$ actions with rectangular towers there remains some choice in the shapes of the towers.  In particular, the dimensions of the rectangles can grow to infinity at different rates.  Because the choice of natural tower shapes is less obvious for $d>1$,  we drop the terminology {\it funny} rank one, and replace it with adjectives describing the choices made in the tower shapes.  If the sequence of towers all are rectangular, we call the action {\it rectangular rank one}.

If $T$ is a rank one $\mathbb Z^d$ action with the property that the sequence of towers are rectangles of uniform bounded eccentricity then $T$ is loosely Bernoulli \cite{JS1}.  While it is not known if there is a rectangular rank one $\mathbb Z^d$ action that is not loosely Bernoulli, the fact that the proof in \cite{JS1} does not readily extend to more general sequences of rectangles, together with the one dimensional result of Ferenczi, suggests that even within the class of rectangular rank one $\mathbb Z^d$ actions, we have the possibility of different dynamical behavior.

In this paper we study the directional entropy of rectangular rank one $\mathbb Z^d$ actions with various growth conditions on the tower shapes.  Directional entropy was introduced by Milnor in
\cite{Mil2}, \cite{Mil}, and has been studied extensively (see \cite{Si},\cite{POsaka},\cite{PIsrael},\cite{PK},\cite{BL}).  Here we show that the shape of the towers plays a role in the possible directional entropies that can occur.   We first show that any rectangular rank one $\mathbb Z^2$ action has at least one zero entropy direction.  In \cite{Ru3} Rudolph shows that given any $h>0$ there exists a rectangular rank one $\mathbb Z^2$ action  whose horizontal sub-action is Bernoulli with entropy $h$.   Our main result shows that when the rectangles satisfy a condition we call {\it sub-exponential eccentricity}, which is weaker than bounded eccentricity, then the resulting action has directional entropy zero in every direction. More generally, our result shows that for any sub-exponentially eccentric rank one $\mathbb Z^d$ action, the $\ell$-dimensional entropy is 0 for all $\ell<d$, for all $\ell$ dimensional hyperplanes in $\mathbb Z^d$.   We note that in \cite{DS} the authors construct a family of examples of rank one $\mathbb Z^d$ actions which, in our terminology, are bounded eccentricity rank one actions.  There they compute the entropy of every transformation in the action to be zero.  In the terminology of directional entropy they show that their examples have directional entropy zero in every rational direction.

In order to prove our main theorem we also establish a higher dimensional generalization of a result of Baxter \cite{Bax} which may be of independent interest.  We show that given a rank one $\mathbb Z^d$ action, with a natural restriction on the tower shapes, there is an isomorphic rank one $\mathbb Z^d$ action such that the towers have the same shapes and the sequence of towers form a refining sequence of generators.

The organization of the paper is as follows.  In Section~\ref{s:basicdef} we establish the notation we use in the paper.    In Section~\ref{s:dirent}  we introduce the idea of directional entropy and summarize some results from the theory that we will use.  In Section~\ref{s:rank1} we describe a hierarchy of types of rank one actions based on progressively more restrictive conditions on the sequence of towers.  We also establish properties of rank one actions, highlighting their relationship to this hierarchy. In this section we also state the generalization of the result in \cite{Bax} about refining sequences of partitions, leaving the proof to Section~\ref{last}. Finally in  Section~\ref{s:main} we prove our main results
about the directional entropy of rank one $\BZ^d$ actions.

\section{Basic definitions and a review of directional entropy}\label{s:basicdef}


\subsection{Shapes}

A {\em shape} $R$ is a finite subset of $\BZ^d$.  Let $\vert R\vert$
denote the cardinality of $R$.  For another shape $S$ the inner $S$-boundary of $R$ is defined by
 \begin{equation*}
\partial_S(R)=\bigcup_{R^c\cap(S+\vv)\not=\emptyset}
R\cap \bigg(S+\vv\bigg).
\end{equation*}

A sequence $\cR=\{R_k\}$ of shapes  is a
{\em F\o lner sequence} (see for example \cite{Temp}) if for
any $\vn\in\Z^d$
\begin{equation}
\lim_{k\rightarrow\infty}
\frac{|R_k\triangle (R_k+\vn)|}{|R_k|}=0,\label{folner}
\end{equation}
or equivalently if for any shape $S$,
\begin{equation}
\lim_{k\rightarrow\infty}
\frac{|\partial_S(R_k)|}{|R_k|}=0.\label{vanhove}
\end{equation}
This follows from
the fact that $\partial_S(R_k)\subseteq\bigcup_{\vn\in S-S}R_k\triangle (R_k+\vn)$.


\subsection{Partitions}

Let $(X,\mathcal M,\mu)$ be a Lebesgue probability space and let $L$
be a finite set.  A {\em partition} $\cP$ with {\em alphabet} $L$ is a
measurable function $\cP:X\rightarrow L$.  Equivalently, we think of
$\cP=\{P_a=\cP^{-1}(a):a\in L\}$ as a finite labeled collection of pairwise
disjoint measurable sets (called {\em atoms}) so that $\cup_{a\in
  L}P_a=X$.  We write $P(x)$ for the unique atom $P_a\in\cP$ so that
$\cP(x)=a$ (i.e., $P(x):=\cP^{-1}(\cP(x))$).

For two partitions with alphabet $L$ we define
\begin{equation*}
d(\cP,\cQ)=\sum_{a\in L} \mu(P_a\triangle Q_a).
\end{equation*}
It is well known that $d$ is a complete metric on the space of all partitions with a fixed alphabet.  In particular, this space is essentially a closed subset of $L^1(X,\mu)$.

For a sequence $\cP_k$ of partitions, where $L_k$ is the alphabet of $\cP_k$, we say
$\cP_k\rightarrow\epsilon$ if for any $A\in\mathcal A$,
there exists $I_k\subseteq L_k$ such that
\begin{equation*}
\lim_{k\rightarrow\infty}
\mu((\cup_{a\in I_k}P_a)\triangle A)=0.
\end{equation*}

Let $\cP$ and $\cQ$ be partitions with alphabets $L$ and $M$.  We say
$\cP\le \cQ$ if each $P_a\in \cP$ is a union of elements of $\cQ$:
$P_a=\cup_{b\in I}Q_b$ for some $I\subseteq M$. We define $\cP\vee\cQ$
to be the partition with atoms $P\cap Q$, where $P\in\cP$ and
$Q\in\cQ$. The alphabet for $\cP\vee\cQ$ is $L\times M$.

\subsection{Towers}

Let $T$ be a free measure preserving $\BZ^d$ action on a
Lebesgue probability space $(X,{\mathcal A},\mu)$. Let
$R\subseteq\BZ^d$ be a shape and let $B\in \cA$,
$\mu(B)>0$, satisfy
\begin{equation*}
T^{\vv_1}B\cap T^{\vv_2}B=\emptyset{\rm\ for\ all\ }\vv_1,\vv_2\in R{\rm\ with\ } \vv_1\not=\vv_2.
\end{equation*}
Let $E=(\cup_{\vv\in R}T^\vv B)^c$.
We call the partition $\cP=\{T^\vv A:\vv\in R\}\cup\{E\}$  a
{\em Rohlin tower}, or more specifically, a
 {\em $T$-tower} with {\em shape}-$R$ and {\em base} $B$.
The sets $T^\vv B$, $\vv\in R$ are called the {\em levels} of the tower and $E$
 is called the {\em error set}.
The tower partiton $\cP$ is defined by labeling the level $T^\vv B$ with $\vv$, and labeling the set
$E$ with $\vec\epsilon=\frac{1}{2}(1,1,\dots,1)$. Thus
the alphabet of $\cP$ is
$L=R\cup\{\vec\epsilon\}$.

\subsection{Directional Entropy}\label{s:dirent}

In this section we state only those definitions and results about the entropy of partitions and directional entropy that are necessary for our work.  We refer the reader to \cite{Mil},\cite{Mil2},\cite{BL},\cite{PIsrael},\cite{PK} for a detailed development of the theory of directional entropy.

We define the
{\em entropy} of a partition $\cP$ with alphabet $L$ by
\begin{equation*}
H(\cP)=\sum_{a\in L} -\mu(P_a)\log(\mu(P_a)),
\end{equation*}
where we define $0\log(0)=0$. 


\begin{lem}\label{shieldslem}{\em (\cite{Sh}, Lemma I.6.8)}
Suppose $\cP$ is a partition with alphabet $L$.
Let $M\subseteq L$ and let $\beta=\mu(\cup_{b\in M}P_b)$.
Then
\begin{equation*}
-\sum_{b\in M}\mu(P_b)\log\mu(P_b)\le\beta\log|M|-\beta\log\beta.
\end{equation*}
\end{lem}

Let $V$ be an $n$-dimensional subspace of $\BR^d$, with $1\le n< d$,  and let
$V^\perp$ be its orthogonal complement. Let $Q$ be the unit cube in $V$ and let
$Q'$ be the unit cube in $V^\perp$ centered at $\vec 0$.
Let
\begin{equation}\label{tiltbox}
S(V,t,m)=tQ+mQ',
\end{equation}
and for a partition $\cP$
let
\begin{equation}
\cP_{V,t,m}=\bigvee_{\vw\in S(V,t,m)\cap\BZ^d}T^{-\vw}\cP,
\end{equation}
The following  definitions are essentially due to Milnor \cite{Mil} and
closely match the definitions given in \cite{BL}. First define
\begin{equation}\label{dent3}
h_n(T,V,\cP,m )=\limsup_{t\rightarrow\infty}\frac1{t^n}
H(\cP_{V, t,m}).
\end{equation}
Then put
\begin{equation}\label{dent31}
h_n(T,V,\cP)=\sup_{m>0}h_n(T,V,\cP,m ).
\end{equation}
We call
\begin{equation}
h_n(T,V)=\sup_{\cP{\rm\ finite}} h_n(T,\cP,V)
\end{equation}
the $n$-dimensional entropy of $T$ in direction $V$.  In the case where $n=1$ and $V$ is the subspace spanned by a vector $\vec v$, we simply call $h_1(T,V)$ the directional entropy of $T$ in direction $\vec v$.

In the case $n=d$, we have $V=\BR^d$, and we let
$h_d(T,\BR^d)=h_d(T)$ denote the usual $d$-dimensional entropy.
In particular, let $S(t)=[0,t]^d$, and for a partition
$\cP$ let
\begin{equation*}
\cP_t=\bigvee_{\vw\in S(t)\cap\BZ^d}T^{-\vw}\cP,
\end{equation*}
and let
\begin{equation*}
h_d(T,\cP)=\lim_{t\rightarrow 0} \frac1{t^d}
H(\cP_t).
\end{equation*}
Then as usual
\begin{equation*}
h_d(T)=\sup_{\cP{\rm\ finite}} h_n(T,\cP).
\end{equation*}

The next two results are useful tools to compute directional entropy.

\begin{lem}[Milnor, \cite{Mil}]
\begin{equation}\label{dent4}
h_n(T,V,\cP,m)=\lim_{t\rightarrow\infty}\frac1{t^n}
H(\cP_{V,t,m})
\end{equation}
and
\begin{equation}
 h_n(T,V,\cP)=\lim_{m\rightarrow\infty}h_n(T,V,\cP,m).
\end{equation}
\end{lem}

\begin{lem}\label{seqlem}{\em(\cite{BL}, Proposition 6.15)}
Let $1\le n\le d$. If $\cP_k\le\cP_{k+1}$ and
$\cP_k\rightarrow\epsilon$ then
\begin{equation}\label{epsilon}
h_n(T,V)=\lim_{k\rightarrow\infty}h_n(T,V,\cP_k).
\end{equation}
\end{lem}

\noindent{\bf Comment.} The hypothesis ``expansive'' is included in
\cite{BL} but is not used in the proof.

We note that it is well known that if $V_1\subseteq V_2$
are subspaces of $\BR^d$ with $n_1=\dim(V_1)<n_2=\dim(V_2)\le d$, then $h(T,V_2)>0$ implies $h(T,V_1)=\infty$.

The following straightforward observation is key to the structure of our later arguments.
\begin{lem}\label{structure}
Suppose $P_k$ is a sequence of partitions with $P_k\le P_{k+1}$  with $P_k\rightarrow\epsilon$ and $V$ is an $n$ dimensional subspace of $\mathbb R^d$.  If for all $k$ and $P=P_k$
\begin{equation}\label{goal}
\lim_{t\rightarrow\infty}\frac1{t^n}H(P_{V,t,m})=0
\end{equation}
for all $m>0$, then $h_n(T,V)=0$.
\end{lem}
\begin{proof}
This follows immediately from (\ref{dent3}), (\ref{dent31}), and Lemma~\ref{seqlem}.
\end{proof}

\section{Rank 1}\label{s:rank1}

The most general definition of a rank one action that we will consider is the following.

\begin{defn}\label{plain}
Let $T$ be a free measure preserving $\mathbb Z^d$ action on a Lebesge probability space $(X,{\mathcal A},\mu)$. We say $T$ is $\cR$ {\em rank one} if $\cR=\{R_k\}$ is a sequence of shapes, and there is a sequence $\cP_k$ of
$T$-towers of shape $R_k$ such that $\cP_k\rightarrow\epsilon$.
We say $T$ is {\em rank one} if it is $\cR$ rank one for some $\cR$.
\end{defn}
\noindent This definition places no restrictions on the shapes of the towers, except that implicitly $\cP_k\rightarrow\epsilon$ implies $|R_k|\rightarrow\infty$.

One classical result about rank one transformations that does not depend on the geometry of the tower shapes $R_k$ is ergodicity.

\begin{prop}
If $T$ is a  rank one action of $\BZ^d$,  then $T$ is ergodic.
\end{prop}

\begin{proof}
This is essentially the same as the well known proof for
rank one  $\mathbb Z$ actions
(see for example \cite{Fr}). Let $A$ be an invariant set of positive measure.
Since $\cP_k\rightarrow\epsilon$
we can, for any
$\epsilon>0$, find a tower $\cP_k$ such that on of the levels of the tower is more than $(1-\epsilon)$ covered by $A$.  The invariance of $A$ guarantees that this property holds for the
entire tower.  Since $\epsilon$ is chosen to be arbitrary, we conclude that $A$ has arbitrarily large measure.
 \end{proof}

Given the generality of our definition it is natural to ask how strange the tower shapes can actually be.  While we do not address this question in this paper, we can easily give a simple example of a rank one $\mathbb Z^2$ action with a one dimensional tower.  This example is also rank one with
non-degenerate
rectangular tower shapes (\cite{KY},\cite{RS4}).

\begin{example}\label{simple}
Consider the irrational rotation $R_\alpha$. Since $R_\alpha$ is rank 1 as
a $\BZ$ action, there exists
a sequence of towers $\cQ_k=\{R_\alpha^n B_k:n=0,1,\dots,\ell_k-1\}$
where  $\cQ_k\rightarrow\epsilon$. Now take $R_\beta$ where $\beta,\frac{\beta}{\alpha}\notin{\mathbb Q}$ and define a free $\BZ^2$ action on the circle by
$T^{(n,m)}=R^{n}_\alpha R^{m}_\beta$. Let $\cF=\{[0,1,\dots,\ell_k]\times\{0\}\}$. Then $T$ is $\cF$ rank 1.
\end{example}

In Definition~\ref{plain} the towers have no a priori relationship to one another.  In constructing examples, though, the towers are usually obtained by ``cutting and stacking'' procedures which yield a refining sequence of tower partitions.  In particular, each tower partition is measurable with respect to all subsequent tower partitions.  The following definition incorporates this structure into the tower partitions.

\begin{defn}
We say $T$ is {\em stacking $\cR$ rank one}, $\cR=\{R_k\}$, if $T$ is $\cR$ rank one for a sequence $\cP_k$ of
$T$-towers with shape $R_k$ that
also satisfies $\cP_{k+1}\ge\cP_{k}$.
\end{defn}

Suppose $T$ is a stacking $\cR$ rank one action with
$\cR=\{R_k\}$ such that  $\cup(R_k-R_k)=\mathbb Z^d$. Then it is easy to see that $T$ is isomorphic to an action $T_1$ costructed by a ``cutting and stacking" construction using the shapes $R_k$ (see for example \cite{PR} for a formal definition of such a construction in the case where $d=2$).

Two other classical results about rank one transformations can be extended to the case of stacking rank one actions, again without placing any further restrictions on the shapes in $\cR$.

\begin{thm}
If $T$ is a stacking rank one $\BZ^d$ action  then $T$ has simple spectrum.
\end{thm}
\begin{proof}
The argument in Baxter \cite{Bax} for the case $d=1$ remains valid
in the case $d>1$.
For $f\in L^2(X,\mu)$, let $U_T^\vv f(x)=f(T^\vv x)$. The cyclic subspace generated by $f$, denoted
$\cH(f)$, is the closure of the span of
$\{U_T^\vv f:\vv\in\BZ^d\}$. Simple spectrum
means that there exists $f$ so that
$\cH(f)=L^2(X,\mu)$. Since $T$ is stacking rank 1, there
exist an infinite sequence $F_k=T^{\vv_k}B_k$ of pairwise
disjoint $T$-tower levels.  Let
$f=\sum\chi_{F_k}$. Baxter's argument shows that
this function $f$ satisfies $\cH(f)=L^2(X,\mu)$. The argument does not depend on the dimension of the acting group.
\end{proof}
An immediate corollary of this result is the following.
\begin{cor}
If $T$ is a stacking rank one $\BZ^d$ action then $h_d(T)=0$.
\end{cor}
\begin{proof}
Again, Baxter's argument from \cite{Bax} holds with no changes. If
 $h_d(T)>0$, then by Sinai's theorem for $\BZ^d$ actions (see \cite{OW}) it follows that $T$
has a Bernoulli factor $T'$.
Any Bernoulli $\BZ^d$ action $T'$ has countable Lebesgue spectrum, and in particular, non-simple spectrum.This would imply that $T$ has nonsimple spectrum.
\end{proof}

In the case $d=1$, Baxter shows that rank one transformations with interval tower shapes are stacking rank one transformations, also with interval tower shapes.  The fairly degenerate towers in Example~\ref{simple} can also easily be chosen to be stacking.  In order to generalize Baxter's proof to $\BZ^d$, $d>1$, however, we need to impose one extra condition on the
sequence  $\cR$ of shapes. The class of rank one actions that we consider is the following.

\begin{defn}
We say a $\BZ^d$ action
$T$  is {\em F\o lner rank one} if it is $\cR$ rank one for some
F\o lner sequence $\cR$ of shapes.
\end{defn}

The following result generalizes Baxter's $d=1$ result to the case of F\o lner rank one $\BZ^d$ actions, $d>1$, adding one additional feature:  the tower shapes can be exactly preserved.

\begin{thm}\label{unitile}
Let $\cR=\{R_k\}$ be a F\o lner sequence in $\BZ^d$ with
${\vec 0}\in R_k$ for all $k$.
Let
$T$ be an $\cR=\{R_k\}$-rank one $\BZ^d$
action.
 Then there exists a sequence of $T$-towers  $\cQ_k$  with  shape
$R_k$ so that $\cQ_k\le\cQ_{k+1}$ for all $k$ and $\cQ_k\rightarrow\epsilon$. In particular, $T$ is stacking $\cR$ rank one.
\end{thm}

The proof of Theorem~\ref{unitile} appears in Section~\ref{last}.

\section{Growth conditions and directional entropy}\label{s:main}

In this section we consider rectangular rank one $\BZ^d$ actions.
Let $\vn=(n_1,n_2,\dots,n_d)\in\BZ^d$. We say $\vec n\ge\vec m$ if
$n_i\ge m_i$ for all $i=1,\dots, d$. In this case we define a
shape, called a {\em rectangle}, by
\begin{equation}
[\vm,\vn]=\{\vv\in\mathbb Z^d: \vm\le\vv\le\vn\}.
\end{equation}
This definition of a rectangle is sufficiently general that the degenerate tower  shapes in Example~\ref{simple} are also rectangles.  The definition given below of rectangular rank one actions includes a condition on the rectangles which guarantees that they will have the same dimension as the acting group.  The condition also provides structure which is both natural from the point of ergodic theory and necessary for some of our arguments.

\begin{defn}
We say $T$ is {\em rectangular rank one} if it is $\cR$ rank one for a
F\o lner sequence $\cR$ of rectangles.
\end{defn}

Note that sequence $\cR=\{R_1,R_2,\dots\}$ of rectangles
$R_j=[\vm_j,\vn_j]$ in $\BZ^d$ is a F\o lner sequence if and only if
for any $\vec w>\vec 0$ one has $\vec w_j=(w^j_1,w^j_2,\dots,w^j_d):=\vn_j-\vm_j>\vec w$
for all $k$ sufficiently large (we say
${\vec w}_j\rightarrow+\infty$).
The first result requires no additional restrictions on the geometry of the rectangles.

\begin{thm}\label{all}
Let $T$ be a rectangular rank one $\mathbb Z^d$ action.  Then there is a one dimensional subspace $V$ of $\mathbb R^d$ so that $h_1(T,V)=0$.
\end{thm}

On the other hand, given Rudolph's example \cite{Ru3}, we know that there do exist rectangular rank one actions with at least one direction with positive directional entropy.  Rudolph's construction requires the long sides of the rectangles to grow super-exponentially as a function of the short sides.  Our next result shows that one can't have directional entropy in the absence of exponential growth of the longest side relative to the shortest side.


Given a sequence $\cR$ of rectangles let
\begin{equation}
s_j=\min_{i=1,\dots,d}w_i^j\text{\ and\ }\ell_j=\max_{i=1,\dots,d}
w_i^j.
\end{equation}

\begin{defn}
We say that a F\o lner sequence $\cR$ of rectangles has
{\em subexponential eccentricity} if
\begin{equation}\label{subexp}
\limsup_{j\rightarrow\infty}\frac{\log(\ell_j)}{s_j}<\infty.
\end{equation}
\end{defn}

The following is our main result.

\begin{thm}\label{main}  Let $\cR$ be a F\o lner sequence of rectangles with subexponential eccentricity.  If  $T$ is an $\cR$ rank one $\BZ^d$ action,
 then $h_n(T,V)=0$ for each $n$-dimensional subspace $V$, for all $1\le n\le d$.
\end{thm}

\noindent The geometric idea underlying both proofs is the same.  In the next section we describe the general set up and prove two key lemmas used in both proofs.  This will help make the differences in the arguments, as well as the role of sub-exponential growth, clearer.

\subsection{Geometric preliminaries}
For any rectangular rank one $\BZ^d$ action $T$ we can find a F\o lner sequence $\cR=\{R_k\}$ of
rectangles, and towers  $\cP_k$ of shape $R_k$, such that $\cP_k\rightarrow\epsilon$.
By Theorem~\ref{unitile} we may assume that
\begin{equation}\label{imp}
\cP_k\le\cP_{k+1}.
\end{equation}

Let $V$ be an $n$ dimensional subspace of $R^d$.  By Lemma~\ref{structure} to prove $h_n(T,V)=0$ we know it suffices to show (\ref{goal}) for $P=P_k$ for all $k\ge 1$ and $m>0$.


Fix $m>0$ and $k\ge 1$ and set $\cP=\cP_k$.  Recall that for all $t$, the alphabet of $\cP_{V,t,m}$ will be
 \begin{equation*}
\bigg(R_k\cup\{\vec e\}\bigg)^{S(V,t,m)}.
\end{equation*}

Fix some $j>k$ and then fix $t$ (typically much smaller than $s_j$).
 Call a level
$T^{\vec v}B_j$ of $\cP_j$
{\it good}
if $\vec v\notin\partial_{S(V,t,m)}(R_j)$. Good levels
have the property that $\vec v+S(V,t,m)\subseteq R_j$
for all $\vec v\in R_j$. It follows from (\ref{imp}) that these are precisely those levels of $\cP_j$ that are contained in a single atom of $\cP_{V,t,m}$.  In other words, for all points $x,y$ in such a level, $P_{V,t,m}(x)=P_{V,t,m}(y)$.

Define a set $Y_j\subseteq X$ to be the complement of the union of good levels of $\cP_j$, together with the error set $E_j$.
 Let
\begin{equation*}
\cP^*_j(x)=\begin{cases} \cP_j(x) \text{ if $x\in Y_j^c$, and}\\
* \text{ if $x\in Y_j$}
\end{cases},
\end{equation*}
and let $\cQ_j=\cP_j^*\vee\cP_{V,t,m}$.
Clearly
\begin{equation}
H(\cP_{V,t,m}) \le H(\cQ_j).\label{wkest1}
\end{equation}

To compute $H(\cQ_j)$ we let $G_j$ be the set of atoms of $\cQ_j$ that do not contain the symbol $*$.  Then
\begin{eqnarray}
H(\cQ_j)&=& -\sum_{P\in G_j}\mu(P)\log\mu(P) \label{wkest2}\\
&&\ \ \ \ \ \ \ \ \ \ \ \ -\sum_{P\in G_j^c}\mu(P)\log\mu(P)\label{wkest3}
\end{eqnarray}

The next two lemmas give estimates on (\ref{wkest2}) and (\ref{wkest3}).

\begin{lem}\label{good}
\begin{eqnarray}
-\sum_{P\in G_j}\mu(P)\log\mu(P)&\le&-\log(1-\mu(E_j))+\sum\log w_i^j. \label{wkest6}
\end{eqnarray}
\end{lem}
\begin{proof}
Note first that if a point lies in an atom of $G_j$, the point lies in a good level of $\cP_j$.  Therefore, a generous upper bound for the cardinality of $G_j$ is obtained by assuming that every level gives rise to a distinct atom of $P(V,t,m)$.
Thus we have
\begin{equation*}
|G_j|\le|R_j|-|\partial_{S(V,t,m)}R_j|\le\prod_{i=1}^d w_i^j.
\end{equation*}
On the other hand, for $P\in G_j$
\begin{equation*}
\mu(P)=(1-\mu(E_j))\frac{1}{\prod w_i^j}\le \frac{1}{\prod{w_i^j}}
\end{equation*}
and the result follows.
\end{proof}
We note that, unlike the next result, the proof of Lemma~\ref{good} does not depend on our choice of $V$.

\begin{lem}\label{bad}
\begin{equation}\label{wkelabshlds}
-\sum_{P\in G_j^c}\mu(P)\log\mu(P) \le 2\mu(Y_j)\vert S(V,t,m)\vert\log\vert R_k\vert-\mu(Y_j)\log\mu(Y_j)
\end{equation}
\end{lem}
\begin{proof}
Lemma~{\ref{shieldslem}} gives that (\ref{wkest3}) is less than or equal to
\begin{equation*}
\mu(Y_j)\log\vert G_j^c\vert -\mu(Y_j)\log\mu(Y_j).
\end{equation*}
Recall that $P=P_k$ and that the alphabet of $\mathcal Q_j$ restricted to $G^c_j$ is
\begin{equation*}
\{*\}\times\bigg(R_k\cup\{\vec e\}\bigg)^{S(V,t,m)}
\end{equation*}
Thus
\begin{equation*}
|G_j^c|\le (|R_k|+1)^{|S(V,t,m)|},
\end{equation*}
and the result follows.

\end{proof}

In the proof of both theorems we will show that given a particular choice of $V$ of dimension $n$, there is a sequence of times $t_j$ such that
\begin{equation*}
\limsup_{j\rightarrow\infty}\frac{1}{t_j}H(\cQ_j)=0.
\end{equation*}
By (\ref{wkest1}) we will have shown (\ref{goal}).
The previous two lemmas, together with the geometry of $V$ will determine the particular choice of $t_j$.

\subsection{Proof of Theorem~\ref{all}}

Suppose, without loss of generality, that for all $j\ge 1$, the maximal dimension of $R_j$ is in the direction $\vec e_1$, i.e. $\vec w_1^j=\ell_j$.  We will show that (\ref{goal}) holds for $n=1$ where $V$ is the one dimensional subspace of $\mathbb R^d$ spanned by $\vec e_1$.   In this case, for all $t\in\mathbb R$, $S(V,t,m)$ is a rectangle and $\vert S(V,t,m)\vert\le tm^{d-1}$. Therefore, by Lemma~\ref{bad},
\begin{equation}
-\sum_{P\in G_j^c}\mu(P)\log\mu(P)\le2\mu(Y_j)t m^{d-1}\log\vert R_k\vert-\mu(Y_j)\log\mu(Y_j).\label{wkelabshlds}
\end{equation}
Using (\ref{wkest2}), (\ref{wkest3}), (\ref{wkelabshlds}), and Lemma~\ref{good}
\begin{multline}
\frac{1}{t}H(\mathcal Q_j)\le
-\frac{\log(1-\mu(E_j))}{t}+
\sum_{i=1}^d \frac{\log(w_i^j)}{t}\\
+\mu(Y_j)\log\vert R_k\vert m^{d-1}-\frac{\mu(Y_j)\log\mu(Y_j)}{t}.\label{laststep1}
\end{multline}

Set, for each $j\in\mathbb N$,
\begin{equation}\label{weaktimes}
t_j=\sqrt{\ell_j\log(\ell_j)}.
\end{equation}
This implies the next two limits:
\begin{equation}\label{wklims1}
\lim_{j\rightarrow\infty}\frac{t_j}{\ell_j}=0,
\end{equation}
and
\begin{equation}\label{wklims2}
\lim_{j\rightarrow\infty}\frac{\log(\ell_j)}{t_j}=0.
\end{equation}

We now show
\begin{equation*}
\lim_{j\rightarrow\infty}H(\cP_{V,t_j,m})=0
\end{equation*}
by computing the limit of each summand of (\ref{laststep1}) separately.

The first term goes to $0$ due to the fact that $\cP_j\rightarrow\epsilon$.  The second term is clearly bounded above by $d\log(\ell_j)/t_j$,
which goes to $0$ by (\ref{wklims2}).

To compute the limit of the last two terms, we note that for all $t$
\begin{equation*}
\vert\partial_{S(V,t_j,m)}\vert\le t_j\prod_{k=2}^dw_j^k+\sum_{i=1}^dm\prod_{k=1,k\neq i}^d w_j^k
\end{equation*}
 Thus
 \begin{align*}
 \mu(Y_j)&\le\bigg[t_j\prod_{k=2}^dw_j^k+\sum_{i=1}^dm\prod_{k=1,k\neq i}^d w_j^k\bigg]
 \frac1{\vert R_j\vert}+\mu(E_j)\\
&=\frac{t_j}{\ell_j}+\sum_{i=2}^d\frac{m}{w_j^i}+\mu(E_j).
 \end{align*}

Using (\ref{wklims1}) for the first term, the fact that the shapes $\{R_k\}$ form a F\o lner sequence for the second term, and again, the fact  that $\cP_k\rightarrow\epsilon$
for the third, we have
 \begin{equation*}
 \lim_{j\rightarrow\infty}\mu(Y_j)=\lim_{j\rightarrow\infty}\mu(Y_j)\log\mu(Y_j)=0.
 \end{equation*}\qed

\subsection{Proof of Theorem~\ref{main}}
Fix $1\le n<d$ and let $V$ be an $n$-dimensional subspace of $\mathbb R^d$.  Recall that $\ell_j$ denotes the largest dimension, and $s_j$ the smallest dimension of $R_j$. It follows from (\ref{subexp}), the assumption of subexponential eccentricity, that by passing to a subsequence we can assume
\begin{equation}\label{subexp1}
\lim_{j\rightarrow\infty}\frac{\log(\ell_j)}{s_j}=\alpha<\infty.
\end{equation}

%

 For an arbitrary $V$ and $t$,
 \begin{equation*}
\vert S(V,t,m)\vert\le t^nm^{d-n},
 \end{equation*}
so by Lemma~\ref{bad},
\begin{equation*}
-\sum_{P\in G_j^c}\mu(P)\log\mu(P)\le2\mu(Y_j)t^n m^{d-n}\log\vert R_k\vert-\mu(Y_j)\log\mu(Y_j),
\end{equation*}
Again, using this together with (\ref{wkest2}), (\ref{wkest3}), and Lemma~\ref{good}, it
  follows that
 \begin{multline}
\frac{1}{t^n}H(\mathcal Q_j)\le
-\frac{\log(1-\mu(E_j))}{t^n}+
\sum_{i=1}^d \frac{\log(w_i^j)}{t^n}\\
+\mu(Y_j)\log\vert R_k\vert m^{d-n}-\frac{\mu(Y_j)\log\mu(Y_j)}{t^n},\label{laststep}
\end{multline}
and we consider each summand individually as time grows to infinity.

The argument for the first term is the same as in the proof of Theorem~\ref{all}.  We have already seen that choosing a  sequence of times
 $t=t_j$ that satisfy (\ref{wklims2}) is sufficient to guarantee that the second summand goes to $0$.
The computation of $\mu(Y_j)$, however, depends on the choice of $V$, and brings in a new necessary condition for the sequence of times to satisfy.   In particular, the size of $\vert\partial_{S(V,t,m)}\vert$
depends on the shape of $S(V,t,m)$, which in turn depends on $V$.  For an arbitrary $V$  and $t$,
\begin{equation*}
\vert \partial_{S(V,t,m)}R_j\vert\le\sum_{\text{$n$-tuples}}(t+m)^n\prod_{\{i\notin\text{the $n$-tuple}\}}w^j_i
\end{equation*}
and therefore
\begin{eqnarray}
\mu(Y_j)&\le&|\partial_{S(V,t,m)}R_j|\frac{1}{\vert R_j\vert}+\mu(E_j)\nonumber\\
&=&\sum\frac{t+m}{w_i^j}+\mu(E_j).\label{wkYmsr1}
\end{eqnarray}

To complete the proof using the same idea as in the proof of Theorem~\ref{all}, it suffices
to choose a sequence of times $t=t_j$ that simultaneously satisfy (\ref{wklims2}) and
\begin{equation}\label{weaklims1}
\lim_{j\rightarrow\infty}\frac{t_j}{s_j}=0.
\end{equation}
Note that together, these two limits require $t_j$ to grow faster than $\log(\ell_j)$ but slower than $s_j$.  But (\ref{subexp1}), the subexponential growth hypothesis, implies that the sequence of times
\begin{equation*}
t_j=\sqrt{s_j\log(\ell_j)}
\end{equation*}
satisfies both (\ref{wklims2}) and (\ref{weaklims1}).\qed

\section{The proof of of Theorem~\ref{unitile}}\label{last}

This proof follows the ideas in Baxter \cite{Bax}, together with an improvement in Lemma~\ref{needgeom} needed for the $\BZ^d$ case.
The proof applies more or less verbatim to
F\o lner rank one
actions of any amenable group $G$.

We begin with some definitions that will simplify our arguments.
Given shapes $R$ and $J$ in  $\BZ^d$
we say that $J$ is {\em $R$-separated} if
\begin{equation*}
(R+\vv_1)\cap(R+\vv_2)=\emptyset{\rm\ for\ all\ }\vv_1,\vv_2\in J{\rm\ with\ } \vv_1\not=\vv_2,
\end{equation*}
in which case
$R+J:=\cup_{\vv\in J}R+\vv$
 is a disjoint union.
We say that a shape $S$
 is a {\em stacking} of a shape $R$
if there exists an $R$-separated set $J$
so that $R+J\subseteq S$. We call $J$ an {\em $R$ stacking set for $S$}.

The purpose of a stacking is that it tells us how
a tower of shape $S$ can be put together out of
a tower of shape $R$. In particular, suppose $R$ and $S$ are shapes, with
$\vec 0\in R\subset S$, and suppose that
$J$ is an $R$ stacking set for $S$. Given a
tower $\cQ$ with shape $S$ and base $B$, let
$A=\cup_{\vec v\in J}T^{\vec v}B$.
Then $A$ is the base of tower $\cQ_J$ of shape $R$, defined by
$\cQ_J=\{T^{\vec v}A:\vec v\in R\}$,
such that $\cQ_J \le\cQ$.

\begin{lem}\label{metric}
Let $\cQ$ and $\cQ'$ be towers of shape $S$. Let $R$ be a shape
such that $J$ is an $R$ stacking set for $S$.
Then
\begin{equation}
d(\cQ_J,\cQ'_J)\le d(\cQ,\cQ').
\end{equation}
\end{lem}

\begin{proof}
Let $B$ and $B'$ denote the bases of $\cQ$ and $\cQ'$. By definition we have that
$A=\sum_{\vw\in J}T^\vw B$ and $A'=\sum_{\vw\in J}T^\vw B'$ are the
bases of $\cQ_J$ and $\cQ'_J$. Then
\begin{eqnarray*}
d(\cQ_J,\cQ'_J)&=&\ \sum_{\vv\in R}\mu(T^\vv A\triangle T^\vv A')+\mu(E\triangle E')\\ \nonumber
&=&|R|\, \mu(A\triangle A')+\mu(E\triangle E')\\ \nonumber
&=&|R|\, \mu(\cup_{\vw\in J}T^\vw B\triangle \cup_{\vw\in J}T^\vw B')+\mu(E\triangle E')\\ \nonumber
&\le & |J|\,|R|\,\mu(B\triangle B')+\mu(E\triangle E')\le |S|\,\mu(B\triangle B')+\mu(E\triangle E')\\  \nonumber
&=&d(\cQ,\cQ').
\end{eqnarray*}
\end{proof}

Let $\cP$ be a tower of shape $R$.
For $A\in \mathcal A$
let $I$ be the set of $a\in R$ so that
\begin{equation}\label{best}
\mu(A\cap P_a)>\frac{1}{2}\mu(P_a).
\end{equation}
Define $A(\cP)=\cup_{a\in I}P_a$.

\begin{lem}\label{stackup}
Suppose $\cP$ and $\cQ$ are towers of shapes $R$ and $S$ satisfying $\vec 0\in R\subset S$, such that $\cP\le\cQ$. Let  $A$ and $B$ be the base sets of $\cP$ and $\cQ$ and suppose $A(\cQ)\neq\emptyset$.
If $J$ is the maximal set of indices in $S$ that satisfy:
 \begin{equation*}
 \cup_{\vec v\in J} T^{\vv}B\subset A(\cQ)\qquad\text{and}\qquad J\cap\partial_R(S)=\emptyset,
 \end{equation*}
 then $\cQ_J$ is a tower of shape $R$ with base
  $A'= \cup_{\vec v\in J} T^{\vv}B$.
 \end{lem}
 \begin{proof}
  Since $\cQ$ is a tower, it follows from
  (\ref{best}) that the sets $\{T^\vv A':\vv\in R\}$
  are pairwise disjoint. Thus $J$ is $R$ separated and an $R$ stacking set for $S$.
\end{proof}

 Under the hypotheses of Lemma~\ref{stackup}, we define
 $\cP(\cQ)=\cQ_J$. Note that $\cP(\cQ)$ and $\cP$ have the same shape $R$, and that $\cP(\cQ)\le\cQ$.  The next lemma shows that
if we  know that the levels in $\cQ$ approximate the levels in $\cP$ well (or equivalently,
the levels in $\cQ$ approximate the base $A$ of $\cP$ well), then we also know that
$\cP(\cQ)$ is a good approximation of $\cP$.

\begin{lem}\label{needgeom}
\begin{equation*}
d(\cP(\cQ),\cP)\le |R|\cdot \mu(A(\cQ)\triangle A)+
|\partial_R(S)|\cdot \mu(B)+\vert R\vert\mu(E_{\cQ}),
\end{equation*}
where $E_{\cQ}$ denotes the error set of $\cQ$.
\end{lem}

\begin{proof}
Suppose $J$ is the index set in $S$ so that $\cP(\cQ)=\cQ_J$.  Then we have
\begin{eqnarray}
d(\cP(\cQ),\cP)
&=&|R|\,\mu(A(\cQ)\triangle A)\\ \nonumber
&\le &|R|\,\mu\left((A(\cQ)\triangle A)\cup
\bigcup_{\vw\in J\cap\partial_R(S)}T^\vw B\cup E_{\cQ}\right) \\ \nonumber
&=&|R|\,\mu(A(\cQ)\triangle A)+|R|\,
\mu\left(\bigcup_{\vw\in J\cap\partial_R(S)}T^\vw B\cup E_{\cQ}\right) \\ \nonumber
&\le&|R|\mu(A(\cQ)\triangle A)+|R|\,|J\cap\partial_R(S)|\,\mu(B)+\vert B\vert \mu(E_{\cQ}),\nonumber
\end{eqnarray}
where we used the identity
$(A\backslash C)\triangle B\subseteq (A\triangle B)\cup C$.
The result follows since
$|\partial_R(S)|\le|R|\,|J\cap\partial_R(S)|$, which is true since
$J$ is $R$-separated.
\end{proof}

The next result shows that given a sequence of towers, we can construct a new sequence, maintaining the shapes, such that each tower is well approximated by subsequent towers far enough  in the sequence.
\begin{lem}\label{Cauchy}
Let $\cR=\{R_k\}$ be a F\o lner sequence, and $\cP_k$ be a sequence of towers of shapes
$R_k$, such that $0\in R_k$ and
$\cP_k\rightarrow \epsilon$.
Then for any $k\ge 1$ and $\delta>0$, we have for all sufficiently large $\ell>k$ that $R_k\subseteq R_\ell$ and
 $d(\cP_k(\cP_\ell),\cP_\ell)<\delta$.
\end{lem}

 \begin{proof}
 Fix $k\ge 1$ and
 use the F\o lner property to choose $\ell$ large
  enough that $R_k\subseteq R_\ell$.   Let $A_k$ be the base of $\cP_k$, and $E_k$ the error set.  In addition, since
 $\cP_k\rightarrow\epsilon$, for all sufficiently large $\ell$ we have  
\begin{equation*}
\mu(A_k(\cP_\ell)\triangle A_k)<\frac{\delta}{3\,|R_k|},\qquad\text{and}\qquad\mu(E_\ell)<\frac{\delta}{3\vert R_k\vert}.
\end{equation*}
 For such an $\ell$ the towers $\mathcal P=\mathcal P_k$ and $\mathcal Q=\mathcal P_{\ell}$ satisfy the hypotheses of Lemma~\ref{stackup}.  Thus we can apply Lemma~\ref{needgeom} to conclude
\begin{equation}\label{setup}
d(\cP_k(\cP_\ell),\cP_\ell)\le |R_k|\mu(A_k({\cP_\ell})\triangle A) +
|\partial_{R_k}(R_\ell)|\mu(A_\ell)+\vert R_k\vert\mu(E_{\ell}).
\end{equation}
Since
Since $\cP_\ell$ is a tower, for all $\ell$
we have
$\mu(A_\ell)\le 1/|R_\ell|$.   So we have
\begin{equation}
|\partial_{R_k}(R_\ell)|\mu(A_\ell)
\le\frac{|\partial_{R_k}(R_\ell)|}{|R_\ell|}
\end{equation}
and since $\cR$ is a F\o lner sequence this can be made $<\frac{\delta}{3}$ for all $\ell$ sufficiently large.
Using these estimates in (\ref{setup}) we have
\begin{equation*}
d(\cP_k(\cP_\ell),\cP_\ell)<\delta.
\end{equation*}
\end{proof}

We are now ready for the proof of Proposition \ref{unitile}.

\begin{proof}[Proof of Proposition \ref{unitile}]
Let $\delta_k>0$
satisfy
$\sum\delta_k<\infty$.
Using Lemma~\ref{Cauchy}, we can then assume
by passing to a subsequence, that the sequence $\cP_k$ of $T$-towers satisfies
$\rho(\cP_k(\cP_{k+1}),\cP_{k+1})<\delta_k$ for all $k$.

We will now
define a doubly infinite sequence of $T$-towers $\cP_{k,\ell}$, $k\ge 1$, $\ell\ge 0$. We start by putting, for each
$k$, $\cP_{k,0}:=\cP_k$ and
\begin{equation*}
\cP_{k,1}:=\cP_{k,0}(\cP_{k+1,0})=\cP_k(\cP_{k+1}).
\end{equation*}
Note that $\cP_{k,1}$ has shape $R_k$ and satisfies
$\cP_{k,1}\le\cP_{k+1,0}$. By Lemma~\ref{stackup} there exists
a stacking set
$I_k\subseteq R_{k+1}\backslash\partial_{R_k}(R_{k+1})$, which, in particular, is
$R_k$-separated.

Now suppose we have defined $\cP_{k,m}$ for all $k$ and for all
$0\le m<\ell$. Then for $k=1,2,\dots$, we define
\begin{equation}\label{stack}
\cP_{k,\ell}=(\cP_{k+1,\ell-1})_{I_k}.
\end{equation}
By induction, $\cP_{k+1,\ell-1}$ is a tower of shape
$R_{k+1}$, so as in our previous discussion, (\ref{stack}) is well defined.  The  result is a tower $\cP_{k,\ell}$ of shape $R_k$ satisfying
$\cP_{k,\ell}\le\cP_{k+1,\ell-1}$.

Repetedly using Lemma~\ref{metric} and equation (\ref{stack}), we have by induction that
\begin{equation}\label{cauchyest}
\rho(\cP_{k,\ell},\cP_{k,\ell+m})\le\sum_{j=\ell}^{\ell+m-1}\delta_j.
\end{equation}
Equation (\ref{cauchyest}) shows $\cP_{k,\ell}$ is a Cauchy sequence
in $\ell$ for each $k$, and we
define $\cQ_k=\lim_\ell \cP_{k,\ell}$. It follows that
$\cQ_k$ is a partition of shape $R_k$. Also
$\cQ_k=(\cQ_{k+1})_{I_k}$, so that $\cQ_k\prec\cQ_{k+1}$.
Moreover, $d(\cQ_k,\cP_k)<\sum_{j\ge k}\delta_k\rightarrow 0$. This
implies $\cQ_k\rightarrow\epsilon$, since
\begin{eqnarray}
\mu(A({\cQ_k})\triangle A) & \le & \mu(A({\cQ_k})\triangle A({\cP_k}))
+\mu(A({\cP_k})\triangle A)
\nonumber \\
& \le &  d(\cQ_k,\cP_k)+\mu(A({\cP_k})\triangle A).
\end{eqnarray}
\end{proof}



\def\cprime{$'$}
\providecommand{\bysame}{\leavevmode\hbox to3em{\hrulefill}\thinspace}
\providecommand{\MR}{\relax\ifhmode\unskip\space\fi MR }
\providecommand{\MRhref}[2]{%
  \href{http://www.ams.org/mathscinet-getitem?mr=#1}{#2}
}
\providecommand{\href}[2]{#2}

 \end{document}